\documentclass{amsart}
\usepackage{amsfonts,amsmath,amssymb,rotating,latexsym}
\usepackage{commath,comment}

\theoremstyle{plain}
\newtheorem{theorem}{Theorem}
\newtheorem{lemma}[theorem]{Lemma}
\newtheorem{corollary}[theorem]{Corollary}

\numberwithin{theorem}{section}
\numberwithin{equation}{section}

\newcommand{\B}{\mathbb{B}}
\newcommand{\C}{\mathbb{C}}
\newcommand{\N}{\mathbb{N}}
\newcommand{\R}{\mathbb{R}}
\newcommand{\T}{\mathbb{T}}
\newcommand{\Z}{\mathbb{Z}}

\newcommand{\cH}{\mathcal{H}}

\newcommand{\GL}{\mathrm{GL}}

\newcommand{\SL}{\mathrm{SL}}
\newcommand{\SU}{\mathrm{SU}}
\newcommand{\U}{\mathrm{U}}
\newcommand{\Spe}{\mathrm{S}}

\newcommand{\im}{\mathrm{Im}}

\newcommand{\wG}{\widetilde{G}}
\newcommand{\wH}{\widetilde{H}}

\newcommand{\ip}[2]{\left< #1,#2 \right>}

\newcommand{\ind}{\mathrm{ind}}
\newcommand{\ev}{\mathrm{ev}}
\newcommand{\id}{\mathrm{id}}

\newcommand{\HBlam}{\mathcal{H}^2_\lambda(\B^n)}
\newcommand{\Llam}{L^2_\lambda(\B^n,\mu_\lambda)}
\newcommand{\Blam}{\mathcal{B}_\lambda}
\newcommand{\Dlambda}{D_\lambda}

\newcommand{\cF}{\mathcal{F}}

\def\sideremark#1{\ifvmode\leavevmode\fi\vadjust{\vbox to0pt{\vss \hbox to 0pt{\hskip\hsize\hskip1em\vbox{\hsize2cm\tiny\raggedright\pretolerance10000  \noindent #1\hfill}\hss}\vbox to8pt{\vfil}\vss}}}

\begin{document}
\title[Restriction Principle and Topelitz Operators]{The Restriction Principle and Commuting Families of Toeplitz Operators on the Unit Ball}

\author{Matthew Dawson}
\address{Centro de Investigaci\'{o}n en Matem\'{a}ticas,  Jalisco s/n, Col. Valenciana, Guajauato, GTO 36240, M\'{e}xico}
\email{matthew.dawson@cimat.mx}

\author{Gestur \'{O}lafsson}
\address{Department of Mathematics, Louisiana State University, Baton Rouge, LA 70803, U.S.A.}
\email{olafsson@math.lsu.edu}

\author{Ra\'{u}l Quiroga-Barranco}
\address{Centro de Investigaci\'{o}n en Matem\'{a}ticas,  Jalisco s/n, Col. Valenciana, Guajauato, GTO 36240, M\'{e}xico}
\email{quiroga@cimat.mx}

\begin{abstract}
    On the unit ball $\B^n$ we consider the weighted Bergman spaces $\HBlam$ and their Toeplitz operators with bounded symbols. It is known from our previous work that if a closed subgroup $H$ of $\widetilde{\SU(n,1)}$ has a multiplicity-free restriction for the holomorphic discrete series of $\widetilde{\SU(n,1)}$, then the family of Toeplitz operators with $H$-invariant symbols pairwise commute. In this work we consider the case of maximal abelian subgroups of $\widetilde{\SU(n,1)}$ and provide a detailed proof of the pairwise commutativity of the corresponding Toeplitz operators. To achieve this we explicitly develop the restriction principle for each (conjugacy class of) maximal abelian subgroup and obtain the corresponding Segal-Bargmann transform. In particular, we obtain
    a multiplicity one result for the restriction of the
    holomorphic discrete series to all maximal abelian subgroups. We also observe that the Segal-Bargman transform is (up to a unitary transformation) a convolution operator against a function that we write down explicitly for each case. This can be used to obtain the explicit simultaneous diagonalization of Toeplitz operators whose symbols are invariant by one of these maximal abelian subgroups.
\end{abstract}

\maketitle


\section{introduction}
In recent decades, one of the principal aims of research on Toeplitz operators on weighted Bergmann spaces over complex bounded symmetric domains has been the study of commuting families of Toeplitz operators.  In particular, one would like to find large families of Toeplitz operators that generate commutative $C^*$-algebras and, when possible, develop explicit formulas for the spectrum of a Toeplitz operator from one of these families in terms of its symbol.

In each of the known examples of maximal abelian $C^*$-algebras generated by commuting familes of Toeplitz operators, one unifying characteristic is that they consist of all Toeplitz operators with symbols invariant under a subgroup of the group of biholomorphisms of the bounded symmetric domain.  For instance, for the case of the unit ball $\B^n$, it has been shown that for each maximal abelian subgroup $H$ of $\SU(n,1)$, the Toeplitz operators with $H$-invariant symbols generate a maximal abelian $C^*$-algebra of operators (see \cite{QVBall1,QVBall2}).  In addition, explicit integral formulas were found (also in \cite{QVBall1,QVBall2}) for the spectrum of a Toeplitz operator in one of these families in terms of its symbol.  These results were proved using ad-hoc techniques, their proofs were rather long and it was perhaps not clear whether a more unifying principle could be used to calculate the spectra.

On the other hand, in representation theory, the weighted Bergmann spaces on a complex bounded symmetric domain $X=G/K$ are well known as spaces which carry the action of scalar-type holomorphic discrete series representations $\pi_\lambda$ of the hermitian Lie group $G$.  Although much of the research in Toeplitz operator theory has not explicitly made use of this connection to representation theory, recent developments suggest that this is a relationship which should be exploited more.

For instance, in \cite{DOQ} it was shown that if $H\subset G$ is a subgroup such that the restriction $\pi_\lambda|_H$ is multiplicity-free, then the Toeplitz operators over the weighted Bergmann space $\HBlam$ with $H$-invariant symbols form a commuting family.  Many examples of subgroups $H\subset G$ that give rise to multiplicity-free restrictions can be found using the method of \textit{visible actions} developed by Kobayashi (see \cite{K05,K07,K08}; the earlier paper \cite{FE} by Faraut and Thomas was also important in the development of this theory).  For instance, any symmetric subgroup $H \subset G$ gives rise to a multiplicity-free restriction (\cite{K08}).  Furthermore, it was also shown in \cite{DOQ} that if $H\subset G$ is a compact subgroup, then the $H$-invariant symbols give rise to a commutative family of Toeplitz operators if and only if $\pi_\lambda|_H$ is multiplicity-free.  One interesting consequence of these results was that in the case of higher-rank symmetric domains, these families do not correspond to symbols that are invariant under maximal abelian subgroups, as had previously been conjectured, but rather to subgroups of the group of biholomorphisms of the symmetric domain that admit multiplicity-free restrictions of the corresponding scalar-type holomorphic discrete-series representations.

Nevertheless, these results did not include a calculation of the spectra of the corresponding Toeplitz operators.  In \cite{DQ}, such formulas were found for the maximal compact subgroup of any hermitian Lie group using techniques from representation theory and Jordan algebras.  See also \cite{QGrudsky} for similar representation-theoretic calculations for the case of symbols invariant under a maximal torus in $\SU(n,1)$.  However, it was still not clear how to extend this technique to the case of Toeplitz operators with symbols invariant under noncompact subgroups.

In this paper we use the generalized Segal-Bargmann transform to derive a very general formula for the Toeplitz operators acting as convolution operators on certain $L^2$-spaces of functions or sections of a line bundle. In case the symbols are invariant under an abelian group this can be used to find the spectrum of the Toeplitz operators.
Most of the arguments hold for general bounded domains, but here we carry out the details for the unit ball $\B^n$, as one of our main interests in the present work is to show how representation theory and abstract harmonic analysis give a unified way to view and attack these problems. On the way, we simplify the proofs and ideas. This is done by using the restriction  principle (see  \cite{O00,OO96} and  \cite{OZ}) to construct a Segal-Bargmann transform by way of the polar decomposition of a restriction operator.  In order to explicitly calculate the polar decomposition, the square root of a positive-definite operator must be taken.  For the maximal abelian subgroups of $\widetilde{\SU(n,1)}$, this can be done by using elementary Fourier transform methods.

For complex bounded symmetric domains of higher rank, the same techniques should also work for the case of any symmetric subgroup $H$ of any hermitian Lie group $G$ for which at least one orbit of the group $H$ on the complex bounded domain $G/K$ admits an injective restriction operator.  In fact, our Theorem~\ref{thm:mainspectraltheorem} is still valid for such symmetric subgroups. In order to take the square root necessary to explicitliy write the Segal-Bargmann transform, the spherical Fourier transform is expected to play the same role as the classical Fourier transform in this work.  We hope to carry out the details and calculate in the future paper \cite{DOQ17}

We begin with a review of Bergman spaces, the holomorphic discrete series, and the restriction principle in Section~\ref{sec:bergmanspace}.  Next, we briefly review previous results on Toeplitz operators in Section~\ref{sec:toeplitzintro}.  In Section~\ref{sec:maxabelian}, we explicitly calculate the Segal-Bargmann transform for restriction to orbits of maximal abelian subgroups of $\SU(n,1)$.  Finally, in Section~\ref{sec:spectrum}, we use the Segal-Bargmann transform to provide formulas for the spectrum of a Toeplitz operator with symbols that are invariant under such subgroups.

\section{The group $\SU(n,1)$ and the Bergman space}
\label{sec:bergmanspace}
In this section we collect basic facts about the action of the group $G=\SU (n,1)$ on the unit ball
$\B^n=\{z\in \C^n\mid |z|<1\}$. Then we review the action of $G$ and its universal
covering group $\wG$ on the Bergman spaces $\HBlam$ of holomorphic functions on $\B^n$.

\subsection{The action of $\SU(n,1)$ on $\B^n$}
The group $G=\SU (n,1)$ is the subgroup of $\SL (n+1,\C)$ that preserves the sesquilinear form
\[ \ip{J_{n,1}z}{w}
    =  - z_{1}\overline{w}_{1} - \dots - z_{n}\overline{w}_{n}+z_{n+1}\overline{w}_{n+1},
\]
where
\[
    J_{n,1} =
        \begin{pmatrix}
        -I_n & 0 \\
        0 & 1
        \end{pmatrix}.
\]
We write a matrix  $A$ in $M_{n+1}(\C)$ in block form as
\begin{equation}\label{eq1}
  A = A(a,v,w,d) =:
  \begin{pmatrix}
    a & v \\
    w^t & d
  \end{pmatrix} ,
\end{equation}
where $a\in M_n(\C)$, $v,w\in \C^n$ and  $d\in \mathbb{C}$. Then, $A$ is in $G$ if and only if $\det A =1$ and $AJ_{n,1}A^*=J_{n,1}$.  A simple calculation gives
\begin{equation}\label{eq:inverse}
    A^{-1} = J_{n,1}A^*J_{n,1} =
        \begin{pmatrix}
            a^* & -\overline{w} \\
            -\overline{v}^t & \overline{d}
        \end{pmatrix}.
\end{equation}
This relation gives in particular $v = -\bar d^{-1} a \bar w$.

The group $\SU(n,1)$ has some important subgroups which we will discuss later. Here we will only
define the maximal compact subgroup $K$ corresponding to the Cartan involution $\theta$ of $G$ given by $\theta(A)=(A^*)^{-1}$. We also denote as usual $\U (1)\simeq \T$, the one-dimensional torus.
Then
\[
    K=G^\theta=\Spe(\U(n)\times \U (1))=\left\{\left. k_a =
        \begin{pmatrix} a & 0 \\ 0& 1/\det a \end{pmatrix}\, \right|
            a\in \U (n) \right\}
                    \simeq \U (n)
\]
is a maximal compact subgroup of $G$.

The group $G$ acts transitively on $\B^n$ via the fractional linear transformations
\[
  \begin{pmatrix}
    a & v \\
    w^t & d
  \end{pmatrix}\cdot z = \frac{az+v}{w^t z + d}.
\]
We note that this formula makes sense for any element of $\GL(n+1,\C)$ and we will use that without comments in the sequel.

\subsection{The Cayley Transform}
In the following we will be considering some submanifolds of $\B^n$ given as orbits of certain subgroups of
$G$.  Some of the calculations involved will be simpler in an unbounded realization of  $\B^n$.

For the case of arbitrary dimension $n\in \N$, we define the unbounded domain
\[
    D_n  = \{(z',z_n) \in\C^{n-1}\times\C \,\mid
      \im(z_n) - |z'|^2 > 0 \},
\]
and we let
\[
    C = \begin{pmatrix}
        i \\
        & \ddots \\
        & & i  \\
        & & & -i & i \\
        & & & 1 & 1
        \end{pmatrix},
\]
which defines a biholomorphism $\B^n \rightarrow D_n$ given by
\begin{align*}
  \zeta_k &= i\frac{z_k}{1+z_n}, \quad 1 \leq k \leq n-1, \\
  \zeta_n &= i\frac{1-z_n}{1+z_n}.
\end{align*}
If we consider the subgroup $G^C=C\SU (n,1)C^{-1}$, then $G^C$ realizes the group of biholomorphisms of $D_n$.

\subsection{The Bergman spaces on $\B^n$}
We now discuss the Bergman spaces $\HBlam$ on the unit ball $\B^n$ and the corresponding holomorphic discrete series representation $\pi_\lambda$ which acts irreducibly on $\HBlam$.

The unit ball $\B^n$ can be identified with the unit ball in $\R^{2n}$ and thus can be equipped with the measure $\dif v = 2nr^{2n-1}\dif r \dif\sigma_n$ where $\dif\sigma_n$ is the rotation invariant measure on the sphere $S^{2n-1}$ in $\R^{2n}$ normalized by $\sigma_n (S^{2n-1})=1$. Then $\dif v$ is a rotation invariant probability measure on $\B^n$.

For $\lambda>n$ we define the probability measure
\[
    \dif\mu_\lambda(z) = c_\lambda(1-|z|^2)^{\lambda -n -1} \dif v(z).
\]
where $c_\lambda  = \frac{\Gamma(\lambda)}{n!\Gamma(\lambda - n)}$ is chosen so that $\mu_\lambda$ is a probability measure. We obtain the corresponding Hilbert space $\Llam$ whose norm and inner product will be denoted by using $\lambda$ as a subscript. Then  the weighted Bergman space $\HBlam$ with weight $\lambda$ is the (closed) subspace of holomorphic functions that belong to $\Llam$.

We note that there are at least two standard ways to parameterize the Bergman spaces. In complex analysis it is customary to use $\alpha = \lambda-n-1>-1$ with $A_\alpha (\B^n)$ denoting the corresponding Bergman space. Here we use the parametrization from representation theory as this will better fit into our discussion. In particular, in our notation the ``weightless'' Lebesgue measure corresponds to $\lambda = n+1$ and the invariant measure corresponds to $\lambda = 0$.
 
It is well known that $\HBlam$ is a reproducing-kernel Hilbert space; that is, the point evaluation maps $f\mapsto \ev_z(f)=f(z)$ are continuous functionals and thus, for every $w\in\B^n$ there exists $K_w\in \HBlam$
such that
\[
    f(w) =\ip{f}{K_w}_\lambda
\]
for all $f\in\HBlam$. The function $K_\lambda(z,w)= K_w(z)$ is called the reproducing kernel or Bergman kernel of $\HBlam$.  Furthermore, it is also well known that
\[
    K_\lambda(z,w) = (1 - z^t \overline{w})^{-\lambda} = (1 - \langle z, w\rangle)^{-\lambda}
\]
for all $\lambda > n$ and $z, w \in \B^n$. The orthogonal projection $B_\lambda : \Llam \rightarrow \HBlam$, also known as the Bergman projection, is then given by
\[
    P_\lambda(f)(z) = \int_{\B^n} f(w) K_\lambda(z,w) \dif \mu_\lambda(w)
        = \int_{\B^n} \frac{f(w) (1-|w|^2)^{\lambda-n-1} }{(1 - z^t \overline{w})^\lambda}\dif v(w)
\]

We now consider the map $j_\lambda : G\times \B^n \rightarrow \C$ given by
\[
    j_\lambda
        \left(\begin{pmatrix}
        a & v \\
        w^t & d
        \end{pmatrix},z \right)
        = (w^t z + d)^{-\lambda}.
\]
We note that the right hand side is in fact defined on $G\times \B^n$ only when $\lambda \in\N$. Otherwise, we lift the map to the universal covering $\wG \times \B^n$. If $\lambda$ is rational, then $j_\lambda$ is
well defined on a finite covering of $G$. Recall that $\wG$ acts on $\B^n$ by $g\cdot z=p(g) \cdot z$ where
$p :\wG \to G$ is the covering map. The function $j_\lambda$ satisfies the cocycle relation
\[
    j_\lambda (gh,z)=j_\lambda (g,h\cdot z) j_\lambda(h,z).
\]
 
For this setup, the action of $\widetilde{G}$ on $\B^n$ yields an irreducible unitary projective representation of $\widetilde{G}$ on $\HBlam$ given by
\[
    \pi_\lambda(g) f(z) =
    j_\lambda (g^{-1}, z) f(g^{-1}\cdot z).
\]

\subsection{The restriction principle}
\label{subsec:restrictionprinciple}
In this section we recall some facts about the restriction principle specialised to $\B^n$. Note that all of the results in this section hold for any complex bounded symmetric domain $G/K$.  We refer to \cite{O00,OO96,OZ} for more details.

We recall that a submanifold $M \subset \B^n$ is said to be totally real if the inclusion $M \hookrightarrow \B^n$ can be locally modelled by the natural inclusion $\R^n \hookrightarrow \C^n$. For us it is important that the restriction map $f\mapsto f|_{M}$ is injective, where $f : \B^n \rightarrow \C$ is holomorphic.  If this condition holds, we say that $M$ is \textit{restriction injective}.  We will assume this for the rest of this subsection and call $M$ restriction injective. We will show that for each maximal abelian subgroup $H$ of $G$ one can find a point $z_0\in D$ such that this holds for the orbit $H\cdot z_0$. This is also true for any symmetric subgroup of $G$.

Let $H$ be a closed subgroup of $G$ and denote by $\wH$ the inverse image in $\wG$. We assume that the orbit
$M= H \cdot z_0=\wH\cdot z_0 \subset \B^n$ is restriction injective. Note that we can identify $M \cong \wH/\wH_{z_0} \cong H / H_{z_0}$, where the subindex $z_0$ denotes the corresponding isotropy subgroup which is clearly a compact subgroup in the case of $H$. Thus there exists a measure $\dif\mu$ on $M$ which is invariant under both $H$ and $\wH$.

\begin{lemma}\label{lem:linebundle}
    Let us define $\chi_\lambda : \wH_{z_0} \rightarrow \C$ by $\chi_\lambda (h)=j_\lambda (h, z_0)^{-1}$. Then $\chi_\lambda$ is a unitary character that satisfies
    \[
        j_\lambda (hk,z_0) = j_\lambda (h,z_0) \chi_\lambda (k)^{-1},
    \]
    for all $h\in\wH$, $k\in\wH_{z_0}$.
\end{lemma}
\begin{proof}
    Let $h\in\wH$ and $k,k' \in \wH_{z_0}$. Then
    \[
        \chi_\lambda (kk')=j_\lambda (kk',z_0)^{-1}
            =\left(j_\lambda (k,k'\cdot z_0)j_\lambda(k',z_0)\right)^{-1}
            =\chi_\lambda(k)\chi_\lambda(k')
    \]
    because $k' \cdot z_0 = z_0$. Similarly, we see that
    \[
        j_{\lambda}(hk,z_0)
            =j_\lambda(h,k\cdot z_0) j_\lambda(k,z_0)
            =j_\lambda(h,z_0)\chi_{\lambda }(k)^{-1}.
    \]
\end{proof}

We now consider the induced representation $\rho_\lambda := \ind_{\wH_{z_0}}^{\tilde H} \chi_\lambda$. The
Hilbert space for $\rho_\lambda$ is the space $L^2_{\chi_\lambda}(M, \mu)$ of measurable functions $f: \tilde H \rightarrow \C$ that satisfy
\begin{equation}\label{dovar}
    f(hh_1)=\chi_{\lambda} (h_1)^{-1}f (h)=j_\lambda (h_1,z_0)f(h),
\end{equation}
for all $h \in \wH, h_1 \in \wH_{z_0}$, as well as
\[
    \int_{\tilde H/\tilde H_{z_0}}|f(h)|^2\, \dif\mu (h)=\int_{M}|f(m)|^2\, \dif\mu (m)<\infty,
\]
where we have used that $|f|$ is right $\wH_{z_0}$-invariant and, by abuse of notation, we have identified $|f|$ with a function on $M$. We will use similar abuse of notation for the inner product
\[
    \ip{f}{g}_{\chi_\lambda}=\int_{M} f(m)\overline{g(m)}\, \dif\mu (m)
\]
where $f,g\in L^2_{\chi_\lambda}(M, \mu)$. On the other hand, the space of continuous functions satisfying the right covariance (\ref{dovar}) will be denoted by $C_{\chi_\lambda}(M)$. We recall that the elements in $L^2_{\chi_\lambda}(M,\mu)$ can be viewed as $L^2$-section of a line bundle over $M$.

Let us now define $\Dlambda: \widetilde{H}\rightarrow \C$ by $ \Dlambda(h) = j_\lambda (h,z_0)$, so that we have
\begin{equation}\label{eq:CovD}
    \Dlambda (hh_1)=\chi_\lambda (h_1)^{-1}\Dlambda (h)=\overline{\chi_\lambda (h_1)}\Dlambda (h),
\end{equation}
for $h\in \wH$ and $h_1\in \wH_{z_0}$. We assume that $|\Dlambda | \in L^2(M,\mu)$, which implies that $\Dlambda \in L^2_{\chi_\lambda}(M,\mu)$. Next we define $R : \HBlam \rightarrow C_{\chi_\lambda}(M)$ by
\[
    R(f)(h)=\Dlambda(h) f|_{M}(h\cdot z_0),
\]
for all $h \in \wH$.

According to \cite[Lem. 2.10]{CGO16} the smooth vectors in $\HBlam$ are bounded and hence are mapped into
$L^2_{\chi_\lambda}(M,\mu)$. In particular this holds for all holomorphic polynomials.
Furthermore, we compute for $f \in \HBlam$ and $h, h_1 \in \wH$
\begin{align*}
    R(\pi_\lambda(h_1)f)(h)&=j_\lambda (h,z_0) (\pi_\lambda(h_1) f)(h\cdot z_0)\\
        &=j_\lambda(h,z_0)j_\lambda (h_1^{-1},h\cdot z_0)f(h_1^{-1}h\cdot z_0)\\
        &=j_\lambda(h_1^{-1}h,z_0)f(h_1^{-1}h\cdot z_0)\\
        &=(\rho_\lambda (h_1)Rf)(h)\, .
\end{align*}
Finally, as the point evaluations maps are continuous, it follows that $R$ is closed. Denote the closure of $R(\HBlam)$ in $L^2_{\chi_\lambda}(M,\mu)$ by $\Blam$. Since $R$ is closed it follows that $R^*: \Blam \to \HBlam$ is well defined and we have
\begin{align*}
    R^*f(z) &= \ip{R^* f}{K_{z}}_\lambda \\
        &= \ip{f}{RK_z}_{L^2_{\chi_\lambda}(M,\mu)} \\
        &= \int_{\wH/\wH_{z_0}} f(h)\overline{\Dlambda (h)}K (z,h\cdot z_0)\, \dif\mu (h),
\end{align*}
for all $f \in \Blam$ and $z \in \B^n$.

Let us define
\[
    R_\lambda (h,k):= \Dlambda(h) \overline{\Dlambda(k)} K_\lambda (h\cdot z_0,k\cdot z_0),
\]
where $h, k \in \wH$ and note that for $h_1,k_1\in\wH_{z_0}$ we have

\[
    R_\lambda (hh_1,kk_1)=\overline{\chi_\lambda(h_1)} R_\lambda (h,k) \chi_\lambda (k_1).
\]
In particular, for every $f \in \Blam$ the assignment $k \mapsto f(k)R_\lambda(h,k)$ ($h,k \in \wH$) defines a function on $M$ and we can verify that the following holds.
\begin{lemma}
\label{lem:kernelR}
    Let $f \in \Blam$ be given. Then
    \[
        RR^* f(h)=\int_{M} f(k) R_\lambda (h,k)\, \dif\mu (k),
    \]
    for all $h\in\wH$.
\end{lemma}

Denote by $\sqrt{RR^*}$ the square root of the positive operator $RR^*$. Then there exists a unitary
isomorphism $U_\lambda: \Blam \rightarrow \HBlam$ such that
\[
    R^*=U_\lambda \sqrt{RR^*}.
\]
The map $U_\lambda$ is called the Segal-Bargmann transform. As in \cite{O00,OO96} we now have the following result.

\begin{theorem}[The Segal-Bargman transform]
The Segal-Bargmann transform $U_\lambda  : (\Blam,\rho_\lambda) \to (\HBlam,\pi_\lambda|_{\wH})$ is a unitary $\wH$-isomorphism.
\end{theorem}

\section{Toeplitz operators}
\label{sec:toeplitzintro}

In this section we recall basic facts about Toeplitz operators on the unit ball. For further details we refer to \cite{QVBall1,QVBall2}.

\subsection{Toeplitz operators}
For $\varphi \in L^\infty(\B^n)$ we define the multiplier operator $M_\varphi$ on the space $L^2(\B^n,\mu_\lambda)$ in the usual way
\[
    M_\varphi f(z) = \varphi(z)f(z).
\]
Of course, $M_\varphi$ will typically not define an operator from $\HBlam$ to $\HBlam$. We therefore define the Toeplitz operator $T_\varphi^{(\lambda)}$ with symbol $\varphi$ corresponding to the weight $\lambda > n$ to be the bounded operator
\[
    T^{(\lambda)}_ \varphi: \cH_\lambda \rightarrow \cH_\lambda, \quad
         f  \mapsto P_\lambda M_\varphi f.
\]
In particular, we have
\[
    (T^{(\lambda)}_\varphi f) (z) = \int_{\B^n} \varphi(w ) f(w) K_\lambda(z,w) \dif\mu_\lambda(w),
\]
for $f \in \HBlam$ and $z \in \B^n$. To simplify our notation we mostly write $T_\varphi$ for $T^{(\lambda )}_\varphi$. The operator $T_\varphi$ is bounded and $\|T_\varphi\|\le \|\varphi\|_\infty$. In particular, the assignment $\varphi \mapsto T_\varphi$ defines a bounded operator. Furthermore, it is well known that this assignment is injective.

\subsection{Commutative families of Toeplitz operators and representation theory}
In this section we briefly review some of the results in \cite{DOQ} which connect commutativity of Toeplitz operators with representation theory, in particular restriction of the discrete series representation $\pi_\lambda$ to subgroups of $G$ or $\wG$.

For $\varphi \in L^\infty(\B^n)$ and $g\in \wG$ define $\pi(g)\varphi(z)=\varphi (g^{-1}\cdot z)$. Then \cite[Lem. 3.2]{DOQ} shows that
\[
    \pi_\lambda (g) \circ T^{(\lambda)}_\varphi = T^{(\lambda)}_{\pi(g)\varphi}\circ \pi_\lambda(g)
\]
for all $g \in \wG$.
This shows in particular \cite[Cor. 3.3]{DOQ} that, if $\wH$ is a closed subgroup of $\wG$, then $\varphi$ is $\wH$-invariant if and only if $T_\varphi^{(\lambda)}$ is an intertwining operator for $\pi_\lambda|_{\wH}$.

If $\wH\subset \wG$ is a reasonably well behaved (e.g.~a type I subgroup), then the representation $\pi_\lambda|_{\wH}$ can be decomposed into irreducible representations
\[
    \pi_\lambda|_{\wH} \simeq_{\wH}
    \int_{\widehat{\wH}}^{\oplus} m_\lambda (\rho )\rho \, \dif\nu_\lambda(\rho)
\]
where $\widehat{\wH}$ is the set of equivalence classes of irreducible unitary representations
of $\wH$ and  $m_\lambda : \widehat{\wH}\to \N\cup \{\infty\}$ is a multiplicity function. We say that the representation $\pi_\lambda|_{\wH}$ is multiplicity free if $\pi_\lambda (\rho) \in \{0,1\}$ for all $\rho \in \widehat{\wH}$. For type I groups this is equivalent to the algebra of intertwining operators for $\pi_\lambda|_{\wH}$ being commutative. In that case, if $T :\HBlam \rightarrow \HBlam$ is an intertwining operator, then $T$ decomposes as
\begin{equation}\label{symbol}
    T = \int_{\widehat{\wH}}^{\oplus} \eta_T (\rho)\id_{\cH_\rho}\, \dif\nu_\lambda (\rho),
\end{equation}
where $\eta_T : \widehat{\wH} \rightarrow \C$. Furthermore, every operator of the form (\ref{symbol}) defines a $\wH$-intertwining operator. The set $(\eta_T (\rho))_{\rho}$ is the spectrum of $T$. According to \cite[Thm. 4.2, Thm. 6.4]{DOQ} we have the following result. In what follows, we will denote by $T^{(\lambda)}(\mathcal{A})$ the $C^*$-algebra generated by the Toeplitz operators on $\HBlam$ with symbols in $\mathcal{A}$.

\begin{theorem}[\cite{DOQ}]
    Let $\wH$ be a closed subgroup of $\wG$ and let us denote by $L^{\infty}(\B^n)^{\wH}$ the subspace of $L^\infty(\B^n)$ that consists of the $\wH$-invariant bounded symbols on $\B^n$. Then the following holds:
    \begin{enumerate}
        \item If for some $\lambda > n$ the algebra of bounded $\wH$-intertwining operators for $\pi_\lambda|_{\wH}$ is commutative, then $T^{(\lambda)}(L^{\infty}(\B^n)^{\wH})$ is a commutative $C^*$-algebra. In particular, the result holds if $\wH$ is a type I group, in the sense of von Neumann algebras, and the restriction $\pi_\lambda|_{\wH}$ is multiplicity-free.
        \item If $\wH$ is compact, then $T^{(\lambda)} (L^{\infty}(\B^n)^{\wH})$ is commutative if and only if $\pi_\lambda|_{\wH}$ is multiplicity free.
    \end{enumerate}
\end{theorem}

Assume now that $\wH$ is so that $\pi_\lambda|_{\wH}$ is multiplicity free and that $M=\wH\cdot z_0$ is restriction injective. Let $U :\Blam \rightarrow \HBlam$ be the corresponding Segal-Bargman transform. Then
$U^* : \HBlam \rightarrow \Blam$ defines a unitary intertwining operator and the above discussion implies that
\[
    (\pi_\lambda|_{\wH} ,\HBlam)\simeq (\rho_\lambda, \Blam)
            \simeq \int^\oplus_{\widehat{\wH}} (\rho ,\cH_\rho )\, \dif\nu (\rho )
\]
and the expression
\begin{equation}\label{eq:diag}
    U^* T_\varphi^{(\lambda)} U= \int^\oplus_{\widehat{\wH}} \eta_{\varphi,\lambda}(\rho) \id_{\cH_\rho}\, \dif\mu (\rho),
\end{equation}
gives the diagonalization of the Toeplitz operator $T_\varphi$ and the set $(\eta_{\varphi,\lambda}(\rho))_\rho$ is the spectrum of $T_\varphi^{(\lambda)}$.  We aim to make this diagonalization more explicit.

\section{Restriction of the discrete series to maximal abelian subgroups of $\SU (n,1)$}
\label{sec:maxabelian}
In this section we apply the restriction principle to the maximal abelian subgroups of $\SU(n,1)$. It is not difficult to check the real dimension of the nondegenerate orbits of maximal abelian subgroups of $\SU(n,1)$ is the same as the complex dimension of the domain $\B^n$.   In fact, one can show (see, for instance, \cite{QVBall2}) that these nondegenerate orbits are all lagrangian submanifolds and therefore restriction injective.  Thus, the abstract theory of the restriction principle discussed in Subsection~\ref{subsec:restrictionprinciple} can be applied as long as one can show, for all $\lambda > n$, that the function $D_\lambda$ belongs to $L^2_{\chi_\lambda}(H\cdot z_0)$ for some nondegenerate orbit $H\cdot z_0$ of a maximal abelian subgroup $H\leq\SU(n,1)$.

Let $\Blam$ be the image of the restriction operator $R$ (defined as in Subsection~\ref{subsec:restrictionprinciple}).  Also, we let $\widetilde{H}$ be the subgroup of $\widetilde{\SU(n,1)}$ which covers $H$.  Of course, $H\cdot z_0 = \widetilde{H}\cdot z_0$. Since $\widetilde{H}$ is an abelian group, its regular representation is multiplicity-free, as is any representation induced from a character of a subgroup of $\widetilde{H}$.

In fact, in this section, for each conjugacy class of maximal abelian subgroup $H$ of $\SU(n,1)$ and basepoint $z_0$ of a nondegenerate orbit $H\cdot z_0$, we will write $RR^*$ as a densely-defined convolution operator on $L^2_{\chi_\lambda}(H\cdot z_0)$ in the following way.  First, for each maximal abelian subgroup $H$, we will construct a homomorphic embedding $\widetilde H\hookrightarrow H/H_{z_0} \times \R$.  This will allow us to extend the character $\chi_\lambda$ from $H_{z_0}$ to $\widetilde{H}$ by defining
\[
    \chi_\lambda((h,x)):= e^{2\pi i \lambda x}
\]
for all $h\in H/H_{z_0}$ and $x\in \R$. For each line-bundle section $f$ in $L^2_{\chi_\lambda}(H\cdot z_0)$ (interpreted as a $\chi_\lambda$-equivarient function on $\widetilde{H}$), one can see that $\widetilde{f}:=f\chi_\lambda$ can in fact be factored to a \textit{function} in $L^2(H\cdot z_0)\cong L^2(H/H_{z_0})$, since the line bundle was induced from a central character $\chi_\lambda$.  Then for each maximal abelian subgroup $H$ we will find a function $\phi_H\in L^1(H/H_{z_0})$ such that for $f$ in the domain of $R$,
\[
    RR^* f = \chi_{-\lambda} \cdot(\widetilde{f} * \phi_H),
\]
In fact, we will identify the function $\phi_H$ and explicitly calculate its $H/H_{z_0}$-Fourier transform $\widehat{\phi_H}$.  Since $\phi_H\in L^1(H/H_{z_0})$, it will follow in each case that the operator $RR^*$ is bounded, and hence that $R$ and $R^*$ are bounded as well.

Hence, the closure of the range of $RR^*$ (and thus the closure of the range of $\sqrt{RR^*}$) will be:
\[
  \Blam = \{f\in L^2_{\chi_\lambda}(H\cdot z_0) \mid (\forall  \alpha\in \widehat{H/H_{z_0}}\text{ such that } \widehat{\phi_H}(\alpha)=0)
  \, \,  \cF_{H/H_{z_0}}\widetilde f(\alpha) = 0 \},
\] where $\cF_{H/H_{z_0}} = \widehat{\cdot}$ represents the Fourier transform for the abelian group $H/H_{z_0}$.  Furthermore, the operator $\sqrt{RR^*}$ can then be written as:
\[
\sqrt{RR^*} f = \chi_{-\lambda}\cdot(\widetilde{f} * \omega_H)
\]
for all $f\in L^2_{\xi_\lambda}(H\cdot z_0)$, where $\omega_H$ is defined by $\widehat{\omega_H}(\alpha) = \sqrt{\widehat{\phi_H}(\alpha)}$ for all $\alpha\in\widehat{H}$. Note that $\omega_H$ is guaranteed to exist at least as a tempered distribution on $H/H_{z_0}$ by the boundedness of $\widehat{\phi_H}$.

We will mostly follow the notation in \cite{QVBall2} in the rest of this section.  Also, we will define the Fourier transform on the torus $\T$ by
\[
 \widehat{f}(n)=\int_\T f(z) \overline{z^n} dz = \int_0^1 f(e^{2\pi i x})e^{-2\pi i n x} dx
\]
for all $f\in L^1(\T)$ and $n\in\Z$, where we have normalized the Haar measure on $\T$ to have weight one.  The Fourier transform on $\R$ will be given by the integral
\[\widehat{f}(\xi) = \frac{1}{\sqrt{2\pi}}\int_\R f(x) e^{-ix} dx, \]
where $f\in L^1(\R)$ and $\xi\in\R$.
\subsection{Quasi-Elliptic}
The Quasi-Elliptic abelian subgroup corresponds to the maximal compact torus in $G$:
\[
E(n) = \left\{\left. k_{t,a}= \begin{pmatrix}
                  a t_1 & && & \\
                     & &\ddots & &\\
                     &   &   & a t_n & \\
                     &      &   &  &  a
                            \end{pmatrix}
                      \right|
                     \begin{array}{l}
                      a,t_1,\ldots,t_n\in\T  \\
                      a^{n+1} t_1\cdots t_n = 1
                     \end{array}
                     \right
                               \}.
\]

The subgroup of the simply-connected group $\widetilde{\SU(n,1)}$ which corresponds to $E(n)$ will be denoted by $\widetilde{E(n)}$, which we will identify with the group
\[\widetilde{E(n)} = \{ (t_1,\ldots, t_n,x) \in\T^n \times \R \mid  e^{2\pi i (n+1)x}t_1 \cdots t_n = 1\}\]
with the product
\[
(t,x)\cdot (s,y) = (ts, x+y)
\]
for all $(t,x), (s,y)\in \widetilde{E(n)}$, where $t,s\in\T^n$ and $x,y\in \R$.
The projection map is given by:
\begin{align*}
\widetilde{E(n)} & \rightarrow   E(n)\\
   (t,x) & \mapsto  k_{t,e^{2\pi i x}}.
\end{align*}

Let $z_0=\left(\frac{1}{\sqrt{2n}},\ldots,\frac{1}{\sqrt{2n}}\right)\in\B^n$ and note that
\[
    k_{t,a} \cdot z_0=\frac{1}{\sqrt{2n}}\, (t_1,\ldots ,t_n)
\]
Hence, the action on the $z_0$-orbit is locally free with stabilizers at $z_0$ given by
\begin{align*}
    E(n)_{z_0} &= \{ k_{(t,a)}\mid a^{n+1}=1\} \simeq \Z_n,  \\
    \widetilde{E(n)}_{z_0} &= \left\{\left. (1,\ldots, 1, \frac{k}{n+1}) \, \right|\,   k \in \Z\right\} \cong\Z.
\end{align*}
Finally, we make the identification
\[
   E(n)/E(n)_{z_0} =\widetilde{E(n)}/\widetilde{E(n)}_{z_0}\cong \T^n,
\]
where the projection map is given by:
\begin{align*}
\widetilde{E(n)} & \rightarrow   E(n)/E(n)_{z_0}\\
   (t,x) & \mapsto t.
\end{align*}

We can now explicitly write the restriction operator $R$ for the orbit $E(n)\cdot z_0$.  In fact, for each $q=(t,x) \in \widetilde{E(n)}$, we have that
\[
  \Dlambda(q) = j_\lambda(q,z_0) =
       (e^{2\pi i x})^{-\lambda} = e^{-2\pi i \lambda x} .
\]
Furthermore, $
   |\Dlambda(q)|^2 = 1$ for all $q\in\widetilde{E(n)}$.
    It follows that
$\Dlambda\in L^2_{\chi_\lambda}(E(n)\cdot z_0)$ for all $\lambda\in\R$ and, in particular, for $\lambda > n$.
Thus, the restriction operator is given by
\[
    Rf (t,x)= e^{-2\pi i \lambda x} f\left(\frac{t_1}{\sqrt{2n}}, \ldots, \frac{t_n}{\sqrt{2n}}\right).
\]
for all $(t,x)\in \widetilde{E(n)}$ and $f\in \HBlam$.

Furthermore, if $(t,x)$ and $(s,y)$ are elements of $\widetilde{E(n)}$, then
\begin{align*}
   R_\lambda (h,k) & = \phi_\lambda(h) (1-\langle h\cdot z_0, k\cdot z_0\rangle)^{-\lambda} \overline{\phi_\lambda(k)} \\
         & = e^{-2\pi i \lambda (x-y)} \left(1-\frac{1}{2n}\sum_{i=1}^n t_i \overline{s_i} \right)^{-\lambda} \\
         & = e^{-2\pi i \lambda (x-y)} \left(1-\frac{1}{2n}\sum_{i=1}^n t_i (s_i)^{-1} \right)^{-\lambda}
\end{align*}

Note that if $f\in L^2_{\chi_\lambda}(P(n)\cdot z_0)$, then $\widetilde{f}\in L^2(P(n)\cdot z_0)$, where \[\widetilde{f}(t) = f(t,x)\chi_\lambda(x) = f((t,x)\cdot z_0) e^{2\pi i\lambda x}\]
for any $(t,x)\in \widetilde{E(n)}$.  By Lemma~\ref{lem:kernelR}, we have that
\begin{align*}
  RR^*f(t,x) & = \int_{\R\times\T^n} f(s,y)  e^{-2\pi i \lambda (x-y)} \left( 1
      - \frac{1}{2n} \sum_{i=1}^{n} t_i (s_i)^{-1}\right)^{-\lambda}ds  \\
    & = e^{-2\pi i\lambda x}  (\widetilde{f} * \phi_{E(n)})(t)
\end{align*}
where $\phi_{E(n)}\in L^\infty(E(n))$ is defined by:
\[
   \phi_{E(n)}(t) = \left( 1 - \frac{1}{2n} \sum_{i=1}^{n} t_i\right)^{-\lambda}
\]
for all $t\in\T^n$.  In fact, since $L^\infty(E(n))\subseteq L^1(E(n))$, it follows that $\widehat{\phi_{E(n)}}$ is a bounded function on $\widehat{E(n)} = \Z^{n-1}$ and thus that and thus that $RR^*$ is a bounded operator.

By the generalized binomial theorem and the multinomial theorem, we have that
\begin{align*}
   \phi_{E(n)}(t) & = \sum_{k=0}^\infty
         \left(\begin{matrix}\lambda + k-1 \\ k\end{matrix}
        \right)  \left(\frac{1}{2n}\sum_{i=1}^n t_i\right) ^k\\
            & = \sum_{k=0}^\infty (2n)^{-k}
         \left(\begin{matrix}\lambda + k-1 \\ k\end{matrix}
        \right) \sum_{\begin{matrix}k_1,\ldots, k_n \in\N_0 \\ k_1 +\cdots + k_n =  k\end{matrix}} \frac{k!}{k_1!\cdots k_n!} t_1^{k_1}\cdots t_n^{k_n},
\end{align*}
where the generalized binomial coefficient is defined by
\[\left(\begin{matrix}\lambda + k-1 \\ k\end{matrix}
        \right) := \frac{\Gamma(\lambda+k)}{\Gamma(\lambda)\Gamma(k+1)	} = \frac{1}{kB(\lambda,k)}.
\]

Thus one sees that $\widehat{\phi_{E(n)}}(k_1,\ldots,k_n)\neq 0$ if and only if $k_i\geq 0$ for $1\leq i\leq n$. In fact,
\[\widehat{\phi_{E(n)}}(\alpha) = \left\{\begin{array}{ll}
                       (2n)^{-|\alpha|}   \frac{\Gamma(\lambda+|\alpha|)}{\Gamma(\lambda)}\frac{1}{\alpha_1!\cdots \alpha_n!}  & \alpha_1\geq 0,\ldots, \alpha_n\geq 0   \\
        0          & \text{otherwise,}
                        \end{array}\right.
\]
where $\alpha = (\alpha_1,\ldots,\alpha_n)\in\Z^{n-1}$.
An alternative approach has been presented in \cite{QGrudsky}.

\subsection{Quasi-Parabolic}
For the other maximal abelian subgroups, we describe them first by their action on $D_n$ and as subgroups of $C \SU(n,1) C^{-1}$ before moving back to the $\B^n$ picture.
The quasi-parabolic subgroup is isomorphic to $\T^{n-1}\times\R$ and acts on $D_n$ by:
\[
(t,y)\cdot (z',z_n) = (tz',z_n+y)
\]
where $(z',z_n)\in D_n$ with $z'\in\C^{n-1}$ and $z_n\in\C$, and where $t\in\T^{n-1}$ and $y\in\R$.

As a subgroup of $C \SU(n,1) C^{-1}$, we may write it as:
\[
    \left\{ \left( \left. \begin{matrix}
                 a t_1 \\
                     & \ddots \\
                     &      & a t_{n-1} \\
                     &      &     &  a & ay \\
                     &      &     &    &  a
                            \end{matrix} \right)
          \right| \begin{array}{l}
                    t_i\in\T, a\in\T, y\in\R \\
                    a^{n+1}t_1\cdots t_{n-1}=1
                  \end{array}
                  \right\} 
\]
As a subgroup of $\SU(n,1)$, we obtain:
\[\scalebox{.9}{$
    P(n) = \left\{ p_{t,y,a} = \left( \left. \begin{matrix}
                  at_1 \\
                      \left(1+\frac{|y|^2}{4}\right)^{-\lambda}& \ddots \\
                     &      &a t_{n-1} \\
           &      &     &  a(1+i\frac{y}{2}) & a(i\frac{y}{2}) \\
           &      &     &  a(-i\frac{y}{2})  &  a(1-i\frac{y}{2})
                            \end{matrix} \right)
          \right| \begin{array}{l}
                    t=(t_1,\ldots,t_{n-1})\in\T^{n-1}, \\
                    a\in\T, y\in\R, \\
                    a^{n+1}t_1\cdots t_{n-1}=1
                  \end{array}
                  \right\} 
$}\]
The action of $P(n)$ on the unit ball $\B^n$ is given by
\[
   p_{t,y,a} \cdot (z',z_n) = \left(\frac{2}{-iy z_n + 2-iy}tz', \frac{(2+iy)z_n + iy}{-iy z_n + 2-iy}\right).
\]
In particular,
\[
  p_{t,y,a} \cdot (z',0) = \left(\frac{2}{2-iy}tz', \frac{iy}{2-iy}\right)
\]

When $n=1$, the group $P(1)$ is simply connected and $P(1)_{z_0}$ is trivial, so that $P(1) \cong \widetilde{P(1)}\cong P(1)/P(1)_{z_0}\cong \R$.

When $n>1$, the subgroup of $\widetilde{\mathrm{SU}(n,1)}$ which corresponds to $P(n)$ is the group $\widetilde{P(n)}$, which we will identify with the group
\[
\widetilde{P(n)}=\{(t,y,x) \mid t\in\T^{n-1},\, y,x\in\R,\,e^{2\pi i (n+1) x} t_1\ldots t_{n-1} = 1 \}
\]
with the product
\[
   (t,y,x)\cdot(t',y',x') = (t t',y+y',x+x').
\]
We also make the identification
\[
   P(n)\cdot z_0 \cong P(n)/P(n)_{z_0} \cong \T^{n-1}\times\R
\]
The projection maps are then given by:
\[
\begin{matrix}
\widetilde{P(n)} & \rightarrow & P(n) & \rightarrow & P(n)/P(n)_{z_0}\\
   (t,y,x) & \mapsto & p_{t,y,e^{2\pi i x}} & \mapsto & (t,y).
\end{matrix}
\]
We will work out the details for the case of $n>1$.  We will leave the details of the case $n=1$ to the reader, since one really only needs to remove all references to the parameter ``$x$''.

Now fix $z_0=\left(\frac{1}{\sqrt{2(n-1)}},\ldots,\frac{1}{\sqrt{2(n-1)}},0\right)\in\B^n$.  Then for each $q=(t,y,x) \in \widetilde{P(n)}$, we have that
\[
  \Dlambda(q) = j_\lambda(q,z_0) =
       \left(e^{2\pi i x} (1-i\frac{y}{2})\right)^{-\lambda} = 2^\lambda e^{-2\pi i \lambda x}
      (2-iy)^{-\lambda}.
\]
Furthermore,
\[
   |\Dlambda(q)|^2 = \left|1-i\frac{y}{2}\right|^{-2\lambda} =  \left(1+\frac{|y|^2}{4}\right)^{-\lambda}
\]
 It follows that
$\Dlambda\in L^2_{\chi_\lambda}(P(n)\cdot z_0)$ for all $\lambda > 1/2$, and, in particular, for $\lambda > n\geq 1$.

Furthermore, one sees that, if $h=(t,y,x)$ and $k=(t',y',x')$, then
\begin{align*}
    R_\lambda(h,k) & = \Dlambda(h) (1-\langle h\cdot z_0, k\cdot z_0\rangle )^{-\lambda} \overline{\Dlambda(k)} \\
           & = e^{-2\pi i \lambda (x-x')} (1-iy/2)^{-\lambda} (1+iy'/2)^{-\lambda} \\
           & \cdot
           \left(1-\frac{1}{2(n-1)} \left\langle \frac{t}{1-iy/2},\frac{t'}{1+iy'/2}\right\rangle - \frac{iy/2}{1-iy/2}\cdot \frac{-iy'/2}{1+iy'/2}\right)^{-\lambda} \\
           & = e^{-2\pi i \lambda (x-x')} \\
           & \cdot \left( (1-iy/2)(1+iy'/2)- \frac{1}{2(n-1)} \sum_{i=1}^{n-1} t_i (t'_i)^{-1} -
           (iy/2)(-iy'/2)\right)^{-\lambda} \\
           & = e^{-2\pi i \lambda (x-x')} \left( 1 - \frac{1}{2} i (y-y') -  \frac{1}{2(n-1)} \sum_{i=1}^{n-1} t_i (t'_i)^{-1}\right)^{-\lambda}
  \end{align*}
If $f\in L^2_{\chi_\lambda}(P(n)\cdot z_0)$, then $\widetilde{f}\in L^2(P(n))$, where \[\widetilde{f}(t,y) = f(t,y,x)\chi_\lambda(x) = f(t,y,x) e^{2\pi i\lambda x}\]
for any $(t,y,x)\in \widetilde{P(n)}$.

By Lemma~\ref{lem:kernelR}, we can write the operator $RR^*$ as:
\begin{align*}
  RR^*f(t,y,x)  = & \int_{\R\times\T^{n-1}} f(t',y',z')  e^{-2\pi i \lambda (x-x')} \\
       & \cdot \left( 1 - \frac{1}{2} i (y-y')
      - \frac{1}{2(n-1)} \sum_{i=1}^{n-1} t_i (t'_i)^{-1}\right)^{-\lambda}dt'dy'  \\
     = & e^{-2\pi i\lambda x} ( \widetilde{f} * \phi_{P(n)})(t,y)
\end{align*}
where $\phi_{P(n)}\in L^\infty(P(n))$ is defined by:
\[
   \phi_{P(n)}(t,y) = \left( 1 - \frac{1}{2} i y - \frac{1}{2(n-1)} \sum_{i=1}^{n-1} t_i\right)^{-\lambda}  = 2^{\lambda}\left( 2 - i y - \frac{1}{n-1} \sum_{i=1}^{n-1} t_i\right)^{-\lambda}
\]
By noting that $|\frac{1}{n-1} \sum_{i=1}^{n-1} t_i| \leq 1$, we see that

\begin{align*}
   \left| 2 - i y - \frac{1}{n-1} \sum_{i=1}^{n-1} t_i\right| & = \sqrt{(2-\frac{1}{n-1}\operatorname{Re}  \sum_{i=1}^{n-1} t_i)^2 + (y+\frac{1}{n-1}\operatorname{Im} \sum_{i=1}^{n-1} t_i)^2} \\
          & \geq \sqrt{1+(|y|-1)^2},
\end{align*}
and hence that
\[
|\phi_{P(n)}(t,y)| \leq |1+(|y|-1)^2|^{-\lambda/2},
\]
from which it follows that $\phi_{P(n)}\in L^1(P(n))$ for all $\lambda > 1$ and, in particular, for all $\lambda > n\geq 1$. Thus $RR^*$ is a bounded operator.

After taking the Fourier transform in the $y$ variable, one obtains (using standard Fourier transform tables, see for instance \cite{B}) that $\cF(\phi_{P(n)})(t,\xi) = 0$ if $\xi<0$, while for $\xi>0$ one has:

\begin{align*}
   \cF_y (\phi_{P(n)}) (t,\xi) & = \frac{2^{\lambda}}{\Gamma(\lambda)} \sqrt{2\pi} \xi^{\lambda-1}  \exp\left(-\big(2- \frac{1}{n-1} \sum_{i=1}^{n-1} t_i\big)\xi\right) \\
   & = \frac{2^{\lambda}\sqrt{2\pi}}{\Gamma(\lambda)} \xi^{\lambda-1} e^{-2\xi} \sum_{k=0}^{\infty} \frac{\xi^{k}}{k!}\big(\frac{1}{n-1} \sum_{i=1}^{n-1} t_i \big)^{k} \\
     & = \sum_{k=0}^{\infty}\frac{2^{\lambda+1/2}\sqrt{\pi}}{k!\Gamma(\lambda)} \cdot \frac{\xi^{\lambda+k-1}}{(n-1)^k} e^{-2\xi}  \scalebox{.9}{$\displaystyle \Big(\sum_{\begin{matrix}k_1,\ldots, k_{n-1} \in\N_0 \\ k_1 +\cdots+ k_{n-1} =  k\end{matrix}} \frac{k!}{k_1!\cdots k_{n-1}!} t_1^{k_1}\cdots t_n^{k_n}\Big)$}.
\end{align*}

It follows that the full Fourier transform of $\phi_{P(n)}$ (that is, the Fourier transform of $\phi_{P(n)}$ as a function on $\T^{n-1}\times\R$) is:
\[
    \widehat{\phi_{P(n)}}(\alpha,\xi) = \left\{\begin{array}{ll}
                       \frac{2^{\lambda+1/2}\sqrt{\pi}}{\Gamma(\lambda)} e^{-2\xi} \frac{\xi^{\lambda+|\alpha|-1}}{(n-1)^k}    \frac{1}{\alpha_1!\cdots \alpha_{n-1}!}  & \xi>0,\alpha_1\geq 0,\ldots, \alpha_{n-1}\geq 0   \\
        0          & \text{otherwise,}
                        \end{array}\right.
\]
where $\xi\in\R$ and $\alpha=(\alpha_1,\ldots,\alpha_n)\in\Z^{n-1}$.

\subsection{Quasi-Hyperbolic}
The quasi-hyperbolic abelian subgroup is isomorphic to $\T^{n-1}\times\R^+$ and acts on $D_n$ by:
\[
    (t,r)\cdot (z',z_n) = (rtz',r^2z_n)
\]
where $(z',z)\in D_n$ with $z'\in\C^{n-1}$ and $z\in\C$, and where $t\in\T^{n-1}$ and $r\in\R^+$.

As a subgroup of $C \SU(n,1) C^{-1}$, we may write it as:
\[
        \left\{ \left( \left. \begin{matrix}
               at_1 \\
                     & \ddots \\
                     &      & a t_{n-1} \\
                     &      &     &  ar &   \\
                     &      &     &    &     ar^{-1}
                            \end{matrix} \right)
          \right| \begin{array}{l}
                    t_i\in\T, a\in\T, r\in\R^+ \\
                    a^{n+1}t_1\cdots t_{n-1}=1
                  \end{array}
                  \right\} 
\]

As a subgroup of $\SU(n,1)$, we obtain:
\[
H(n) = \left\{ \left( \left. \begin{matrix}
     at_1 \\
        & \ddots \\
        &      & at_{n-1} \\
      &  &  & a\frac{r+r^{-1}}{2} & a\frac{r-r^{-1}}{2} \\
      &  &  & a\frac{r-r^{-1}}{2} & a\frac{r+r^{-1}}{2}
                            \end{matrix}\right)
          \right| \begin{array}{l}
                    t=(t_1,\ldots, t_{n-1})\in\T^{n-1},\\
                    a\in\T, \, r\in\R^+, \\
                    a^{n+1}t_1\cdots t_{n-1}=1
                  \end{array}
                  \right\} 
\]
After the substitution $s = \log r$, we can write:
\[\scalebox{.95}{$
H(n) = \left\{ h_{t,s,a}=\left( \left. \begin{matrix}
     at_1 \\
        & \ddots \\
        &      & at_{n-1} \\
      &  &  & a\cosh s & a\sinh s \\
      &  &  & a\sinh s & a\cosh s
                            \end{matrix}\right)
          \right| \begin{array}{l}
                    t=(t_1,\ldots, t_{n-1})\in\T, \\
                    a\in\T,\, s\in\R, \\
                    a^{n+1}t_1\cdots t_{n-1}=1
                  \end{array}
                  \right\} 
$}\]
We see that the action of $H(n)$ on the unit disk $\B^n$ is given by
\[
   h_{t,s,a} \cdot (z',z_n) = \left(\frac{tz'}{(\sinh s)z_n + \cosh s}, \frac{(\cosh s) z_n + \sinh s}{(\sinh s)z_n + \cosh s}\right).
\]
In particular,
\[
  h_{t,s,a} \cdot (z',0) = \left(\frac{t}{\cosh s}z', \tanh s\right).
\]
for all $(z',0)\in\B^n$, where $z'\in\C^{n-1}$.

When $n=1$, the group $H(1)$ is simply connected and $H(1)_{z_0}$ is trivial, so that $P(1) \cong \widetilde{H(1)}\cong H(1)/H(1)_{z_0}\cong \R$.

When $n>1$, the subgroup of $\widetilde{\mathrm{SU}(n,1)}$ which corresponds to $H(n)$ is the group $\widetilde{H(n)}$, which we will identify with the group
\[
\widetilde{H(n)}=\{(t,s,x) \mid t\in\T^{n-1},\, s,x\in\R,\,e^{2\pi i (n+1) x} t_1\ldots t_{n-1} = 1 \}
\]
with the product
\[
   (t,s,x)\cdot(t',s',x') = (t t',s+s',x+x').
\]
We also make the identification:
\[
   H(n)\cdot z_0 \cong H(n)/H(n)_{z_0} \cong \T^{n-1}\times\R
\]
The projection maps are then given by:
\[
\begin{matrix}
\widetilde{H(n)} & \rightarrow & H(n) & \rightarrow & H(n)/H(n)_{z_0}\\
   (t,s,x) & \mapsto & n_{t,s,e^{2\pi i x}} & \mapsto & (t,s).
\end{matrix}
\]
As in the Parabolic case, we will work out the details for the case of $n>1$, leaving the details of the case $n=1$ to the reader.

Now fix $z_0=\left(\frac{1}{\sqrt{2(n-1)}},\ldots,\frac{1}{\sqrt{2(n-1)}},0\right)\in\B^n$.  Then for each $q=(t,s.x) \in \widetilde{H(n)}$, we have that
\[
  \Dlambda(q) = j_\lambda(q,z_0) =
       (e^{2\pi i x} \cosh s)^{-\lambda} = e^{-2\pi i \lambda x}
      (\cosh s)^{-\lambda}.
\]
Then
\[
   |\Dlambda(q)|^2 = (\cosh s)^{-2\lambda}
\]
for all $q=(t, s, x)\in\widetilde{H(n)}$.  It follows that
$\Dlambda\in L^2_{\chi_\lambda}(H(n)\cdot z_0)$ whenever $2\lambda>1$, that is, $\lambda > 1/2$.  In particular, this holds for all $\lambda > n\geq 1$.

Furthermore, one sees that, if $h=(t,s,x)$ and $k=(t',s',x')$, then
\begin{align*}
    R_\lambda(h,k)  = & \Dlambda(h) (1-\langle h\cdot z_0, k\cdot z_0\rangle )^{-\lambda} \overline{\Dlambda(k)} \\
            = & e^{-2\pi i \lambda (x-x')} (\cosh s)^{-\lambda} (\cosh s')^{-\lambda} \\
            & \cdot \left(1-\frac{1}{2(n-1)} \left\langle \frac{t}{\cosh s},\frac{t'}{\cosh s'}\right\rangle - \tanh s \tanh s'\right)^{-\lambda} \\
            = & e^{-2\pi i \lambda (x-x')} \left( \cosh(s-s') - \frac{1}{2(n-1)} \sum_{i=1}^{n-1} t_i (t'_i)^{-1}\right)^{-\lambda}
  \end{align*}

If $f\in L^2_{\chi_\lambda}(H(n)\cdot z_0)$, then $\widetilde{f}\in L^2(H(n))$, where \[\widetilde{f}(t,y) = f(t,s,x)\chi_\lambda(x) = f(t,s,x) e^{2\pi i\lambda x}\]
for any $(t,s,x)\in \widetilde{H(n)}$.  By Lemma~\ref{lem:kernelR}, we have that
\begin{align*}
  RR^*f(t,y,x) = & \int_{\T^{n-1}\times\R^+} f(t',y',z')  e^{-2\pi i \lambda (x-x')} \\
       & \cdot  \left( \cosh(s-s')
      - \frac{1}{2(n-1)} \sum_{i=1}^{n-1} t_i (t'_i)^{-1}\right)^{-\lambda}dt'dy'  \\
     = & e^{-2\pi i\lambda x} ( \widetilde{f} * \phi_{H(n)})(t,y),
\end{align*}
where $\phi_{H(n)}\in L^\infty(H(n)\cdot z_0)$ is defined by:
\[ \phi_{H(n)}(t,s) = \left( \cosh(s) - \frac{1}{2(n-1)} \sum_{i=1}^{n-1}t_i\right)^{-\lambda}.
\]
In fact, since $|\frac{1}{2(n-1)}\sum_{i=1}^{n-1} t_i | \leq \frac{1}{2}$,
one sees that
\[
|\phi_{H(n)}(t,s)| \leq \left|\cosh(s)-\frac{1}{2}\right|^{-\lambda}. \]
Thus, $\phi_{H(n)}\in L^1(H(n)\cdot z_0)$ if $\lambda > 1$ and, in particular, for all $\lambda > n \geq 1$.

By once again using the generalized binomial theorem, we obtain
\begin{align*}
   \phi_{H(n)}(t,s)
        & = \sum_{k=0}^\infty \frac{\Gamma(-\lambda+1)}{\Gamma(-\lambda-k+1)\Gamma(k+1)} (\cosh x)^{-\lambda-k}\left(-\frac{1}{2(n-1)} \sum_{i=1}^{n-1} t_i\right)^k \\
        & = \sum_{k=0}^\infty (2(n-1))^{-k}\frac{\Gamma(\lambda+k)}{\Gamma(\lambda)\Gamma(k+1)} (\cosh x)^{-\lambda-k} \left(\sum_{i=1}^{n-1} t_i\right)^k,
\end{align*}
where in going from the first to the second line we use the following easily-verified identity:
\[
 \frac{\Gamma(-\lambda+1)}{\Gamma(-\lambda-k+1)} (-1)^k = \frac{\Gamma(\lambda+k)}{\Gamma(\lambda)}
\]

By using a standard Fourier transform table for $\widehat{\cosh^{-\lambda-k}}$ (see \cite{B}, for instance), we obtain that
\begin{align*}
\widehat{\phi_{H(n)}}(\alpha,\xi)
        =&  \frac{1}{\sqrt{2\pi}}(2(n-1))^{-|\alpha|}  \frac{\Gamma(\lambda+|\alpha|)}{\Gamma(\lambda)\Gamma(|\alpha|+1)}\\
     & \cdot 2^{\lambda+|\alpha|} \mathrm{B}\left(\frac{1}{2}(\lambda + |\alpha| + i\xi),\lambda +|\alpha|-\frac{1}{2}(\lambda + |\alpha| + i\xi)\right)
     \frac{|\alpha|!}{\alpha_1!\cdots\alpha_{n-1}!} \\
      = & \frac{2^\lambda}{ \sqrt{2\pi}(n-1)^{|\alpha|}} \frac{\Gamma(\frac{1}{2}(\lambda+|\alpha|+i\xi))\Gamma(\frac{1}{2}(\lambda+|\alpha|-i\xi))}{\Gamma(\lambda+|\alpha|)}\cdot \\
      & \cdot \frac{\Gamma(\lambda+|\alpha|)}{\Gamma(\lambda)}\frac{1}{\alpha_1!\cdots\alpha_{n-1}!} \\
       = & \frac{2^\lambda}{\sqrt{2\pi}(n-1)^{|\alpha|}} \frac{\Gamma(\frac{1}{2}(\lambda+|\alpha|+i\xi))\Gamma(\frac{1}{2}(\lambda+|\alpha|-i\xi))}{\Gamma(\lambda)}
       \frac{1}{\alpha_1!\cdots\alpha_{n-1}!}
\end{align*}
for all $\alpha\in\Z^{n-1}$ and all $\xi\in\R$ such that $\alpha_1\geq 0,\cdots \alpha_{n-1}\geq 0$.  Furthermore,  $\widehat{\phi_{H(n)}}(\alpha,\xi)=0$ for all $(\alpha,\xi)\in\Z^{n-1}\times\R$ such that $\alpha_i<0$ for some $1\leq i\leq n-1$.
\subsection{Nilpotent}
\label{sec:nilpotent}
The ``nilpotent'' abelian subgroup is isomorphic to $\R^{n-1}\times\R$ and acts on $D_n$ by:
\[
    (b,s)\cdot (z',z_n) = (z'+b,z_n+2i\langle z', b\rangle+s+i|b|^2)
\]
where $(z',z_n)\in D_n$ with $z'\in\C^{n-1}$ and $z_n\in\C$, and where $b\in\R^{n-1}$ and $s\in\R$.

As a subgroup of $C \SU(n,1) C^{-1}$, we may write it as:
\[
    \left\{ \left( \left. \begin{matrix}
               1     &        &           &    & b_1        \\
                     & \ddots &           &    & \vdots     \\
                     &        &         1 &    & b_{n-1}    \\
               2ib_1 & \cdots & 2ib_{n-1} &  1 & s +i|b|^2  \\
                     &        &           &    &  1
                            \end{matrix} \right)
          \right| \begin{array}{l}
                    b_i\in\R,  s\in\R
                  \end{array}
                  \right\}.
\]
Note that it can be shown that each of the above matrices has determinant one.

As a subgroup of $\SU(n,1)$, we obtain:
\[\scalebox{.9}{
         $N(n) =  \left\{ n_{s,b} = \left( \left. \begin{matrix}
         1     &        &           & -ib_1     & -ib_1     \\
               & \ddots &           & \vdots    & \vdots    \\
               &        &     1     & -ib_{n-1} & -ib_{n-1}  \\
            -ib_1 & \cdots & -ib_{n-1} &
               \frac{1}{2}(is-|b|^2)+1 & \frac{1}{2}(is-|b|^2) \\
             ib_1 & \cdots & ib_{n-1} &
               \frac{1}{2}(-is+|b|^2)   & \frac{1}{2}(-is+|b|^2)+1
                            \end{matrix} \right)
          \right| \begin{array}{l}
                    b_i\in\R,\\  s\in\R

                  \end{array}
                  \right\}$
   }
\]
Since this group is simply connected, it is isomorphic to its covering group $\widetilde{N(n)}$ sitting inside the simply-connected group $\widetilde{SU(n,1)}$.

Now fix $z_0=0\in\B^n$.  Then for each $h_{s,b}\in N(n)$, we have that the action on $z_0$ is given by
\[
 n_{s,b}\cdot z_0 = \frac{1}{\frac{1}{2}(-is + |b|^2)+1}\left(-ib_1,\ldots,-ib_{n-1},\frac{1}{2}(is - |b|^2)\right)
\]
Note also that
\[
\Dlambda(n_{s,b}) = j_\lambda(n_{s,b},0)=\left(\frac{1}{2}\left(-is+|b|^2\right) +1\right)^{-\lambda}
\]
Note that this function can be made to be well-defined on $N(n)$ as long as a branch cut is made for the map $x\mapsto x^{-\lambda}$ on the right half-plane of $\C$.

Then
\begin{align*}
\int_{N(n)}|\Dlambda(n_{s,b})|^2 dh
       & = \int_{\R^{n-1}} \int_\R \left|\frac{1}{2}\left(-is+|b|^2\right)
                   +1\right|^{-2\lambda}ds db \\
       & = \left(\frac{1}{2}\right)^{-2\lambda} \int_{\R^{n-1}}
                  \int_\R (s^2+(|b|^2+2)^2)^{-\lambda}ds db \\
       & = 4^{\lambda}\int_{\R^{n-1}}
         (|b|^2+2)^{-2\lambda} \int_\R \left(\left(
         \frac{s}{|b|^2+2}\right)^2+1\right)^{-\lambda}ds db \\
       & = 4^{\lambda}\int_{\R^{n-1}}
         (|b|^2+2)^{-2\lambda+1} \int_\R \left(\left(
         \frac{s}{|b|^2+2}\right)^2+1\right)^{-\lambda}\frac{1}{|b|^2+2}ds db \\
      & = 4^{\lambda}\int_{\R^{n-1}}
         (|b|^2+2)^{-2\lambda+1}db \int_\R (s^2+1)^{-\lambda}ds  \\
\end{align*}
Thus, $\Dlambda\in L^2_{\chi_\lambda}(N(n)\cdot z_0)$ if and only if $\lambda > 1/2$ and $\int_{\R^{n-1}} (|b|^2+1)^{-2\lambda+1}db<\infty$.  In particular, this holds for all $\lambda > n$.  If $n-1=1$, then this last condition is equivalent to $2(-2\lambda+1)<-1$.  If $n-1>1$, then the condition is equivalent to $\int_{\R} (x^2+1)^{-2\lambda+1+(n-2)/2}db<\infty$, which in turn is true if and only if $2(-2\lambda+n/2)<-1$.  To sum everything up, we have that $\Dlambda\in L^2_{\chi_\lambda}(N(n)\cdot z_0)$ if and only if $\lambda>\frac{n+1}{4}$.

Furthermore, one sees that, if $n_{s,b}, n_{s',b'}\in N(n)$, then
\begin{align*}
    R_\lambda(n_{s,b},n_{s',b'})  = &\Dlambda(n_{s,b}) (1-\langle n_{s,b}\cdot z_0, n_{s',b'}\cdot z_0\rangle )^{-\lambda} \overline{\Dlambda(h_{s',b'})} \\
           = &  \left(\frac{1}{2}\left(-is+|b|^2\right) +1\right)^{-\lambda} \left(\frac{1}{2}\left(is'+|b'|^2\right) +1\right)^{-\lambda} \\
            & \cdot \Bigg(1-\frac{1}{\frac{1}{2}(-is + |b|^2)+1} \frac{1}{\frac{1}{2}(is' + |b'|^2)+1} \cdot \\
            & \cdot \bigg(\left\langle -ib,-ib'\right\rangle +\bigg( \frac{1}{2}(is - |b|^2)  \bigg) \bigg( \frac{1}{2}(-is' - |b'|^2)  \bigg)   \bigg)\Bigg)^{-\lambda} \\
            =  &  \Bigg (\left(\frac{1}{2}\left(-is+|b|^2\right) +1\right) \left(\frac{1}{2}\left(is'+|b'|^2\right) +1\right)- \langle b, b' \rangle \\
           & -\left( \frac{1}{2}(is - |b|^2)  \right) \left( \frac{1}{2}(-is' - |b'|^2)  \right) \Bigg )^{-\lambda} \\
           = &  \left( \frac{1}{2} (-i(s-s') + |b-b'|^2)+1 \right)^{-\lambda}
  \end{align*}

 By Lemma~\ref{lem:kernelR}, we have that
\begin{align*}
  RR^*f(s,b) & = \int_{\R\times\R^{n-1}} f(s',b')  \left(
     \frac{1}{2} (-i(s-s') + |b-b'|^2)+1 \right)^{-\lambda}ds'db'  \\
    & = (f * \phi_{N(n)})(t,y)
\end{align*}
for all $f\in L^2_{\xi_\lambda}(N(n)\cdot z_0)$, where $\phi_{N(n)}\in L^\infty(N(n))$ is defined by:
\[ \phi_{N(n)}(s,b)= 2^\lambda \left( -is + |b|^2+2 \right)^{-\lambda}.
\]
In fact, we see that $D_\lambda = \phi_{N(n)}$ and thus the previous calculation for $D_\lambda$ shows that $\phi_{N(n)}\in L^1(N(n)\cdot z_0)$ for all $\lambda > (n+1)/2$ and hence for all $\lambda > n\geq 1$.

One can check that, taking the Fourier transform first in the ``$s$'' variable, one has:
\[
    \cF_s({\phi_{N(n)}})(y, b) = \left\{ \begin{array}{ll}\sqrt{2\pi}2^{\lambda} y^{\lambda-1}e^{(-|b|^2-2)y}, & y >0 \\
         0, & y<0\end{array}\right.
\]
Taking the Fourier transform in the ``$b$'' variable, we obtain:
\[
    \widehat{\phi_{N(n)}}(y, \xi) = \left\{ \begin{array}{ll}\sqrt{2\pi} 2^\lambda y^{\lambda-1}\left(\frac{1}{2y}\right)^{\frac{n-1}{2}}e^{-2y}e^{-|\xi|^2/8y}, & y >0 \\
         0, & y<0\end{array}\right.
\]

\subsection{Quasi-Nilpotent}
\label{sec:quasinilpotent}
The ``quasi-nilpotent'' abelian subgroups are isomorphic to $\T^k\times \R^{n-k-1}\times\R$, where $1\leq k \leq n-2$ (the case $k=0$ reduces to the ``nilpotent'' case and $k=n-1$ corresponds to the quasi-parabolic case) and act on $D_n$ by:
\[
    (t,b,h)\cdot (z',z'',z_n) = (tz',z''+b,z_n+2i\langle z'', b\rangle+s+i|b|^2),
\]
where $(z',z'',z_n)\in D_n$ with $z'\in \C^{k}$, $z''\in\C^{n-k-1}$, and $z_n\in\C$, and where $t\in\T^k$, $b\in\R^{n-k-1}$, and $s\in\R$.

As a subgroup of $C \SU(n,1) C^{-1}$, we may write it as:
\[\scalebox{.8}{$
    \left\{ \left( \left. \begin{matrix}
 t_1  &        &   &    \\
      & \ddots &  \\
      &        & t_k &  \\
        &  &  &  1     &        &           &    & b_1        \\
        &  &  &       & \ddots &           &    & \vdots     \\
        &  &  &       &        &         1 &    & b_{n-k-1}    \\
        &  &  & 2ib_1 & \cdots & 2ib_{n-k-1} &  1 & s +i|b|^2  \\
        &  &  &       &        &           &    &  1
                            \end{matrix} \right)
          \right| \begin{array}{l}
                    b_i\in\R,  s\in\R, t_j \in \T
                  \end{array}
                  \right\}.$
}\]
Note that it can be shown that each of the above matrices has determinant one.

As a subgroup of $\SU(n,1)$, we obtain:
\[\scalebox{.68}{
         $N(k,n) =  \left\{ n_{t,s,b,a} = a\left( \left. \begin{matrix}
 t_1  &        &   &    \\
      & \ddots &  \\
      &        & t_k &  \\

   &  &  &   1     &        &           & -ib_1     & -ib_1     \\
   &  &  &      & \ddots &           & \vdots    & \vdots    \\
   &  &  &     &        &     1     & -ib_{n-1} & -ib_{n-1}  \\
   &  &  &  -ib_1 & \cdots & -ib_{n-1} &
               \frac{1}{2}(is-|b|^2)+1 & \frac{1}{2}(is-|b|^2) \\
   &  &  &   ib_1 & \cdots & ib_{n-1} &
               \frac{1}{2}(-is+|b|^2)   & \frac{1}{2}(-is+|b|^2)+1
                            \end{matrix} \right)
          \right| \begin{array}{l}
                    a\in\T, \\ t_i\in \T, \\ b_i\in\R,\\  s\in\R, \\
                    a^{n+1} t_1\cdots t_k = 1
                  \end{array}
                  \right\}.$
   }
\]
Since this group is simply connected, it is isomorphic to its covering group $\widetilde{N(k,n)}$ sitting inside the simply-connected group $\widetilde{SU(n,1)}$.

Now fix $z_0=\left(\frac{1}{\sqrt{2k}},\ldots,\frac{1}{\sqrt{2n}}, 0, \ldots, 0\right)\in\B^n$, where the first $k$ terms are nonzero.  Then for each $h_{s,b}\in N(n)$, we have that the action on $z_0$ is given by
\[
 n_{t,s,b,a}\cdot z_0 = \frac{1}{\frac{1}{2}(-is + |b|^2)+1}\left(\frac{t_1}{\sqrt{2k}},\ldots, \frac{t_k}{\sqrt{2k}}, -ib_1,\ldots,-ib_{n-1},\frac{1}{2}(is - |b|^2)\right).
\]

The subgroup of $\widetilde{\mathrm{SU}(n,1)}$ which corresponds to $N(k,n)$ is the group $\widetilde{N(k,n)}$, which we will identify with the group
\[
\widetilde{N(k,n)}=\{(t,s,b,x) \mid t\in\T^{k},b\in \R^{n-k-1}, s,x\in\R,\,e^{2\pi i (n+1) x} t_1\ldots t_k = 1 \}
\]
with the product
\[
   (t,s,b,x)\cdot(t',s',b',x') = (t t',s+s',b'+b', x+x').
\]
We also make the identification
\[
   N(k,n)\cdot z_0 \cong N(k,n)/N(k,n)_{z_0} \cong \T^{k}\times\R\times\R^{n-k-1}.
\]
The projection maps are then given by:
\[
\begin{matrix}
\widetilde{N(k,n)} & \rightarrow & N(k,n) & \rightarrow & N(k,n)/N(k,n)_{z_0}\\
   (t,s,b,x) & \mapsto & n_{t,s,b,e^{2\pi i x}} & \mapsto & (t,s,b).
\end{matrix}
\]

Now fix $q=(t,s,b,x)$ in the group $\widetilde{N(k,n)}$.  Then
\begin{align*}
\Dlambda(q) = j_\lambda(q,z_0)=& \left(e^{2\pi i x} \left(\frac{1}{2}(-is+|b|^2) +1\right)\right)^{-\lambda} \\
       = & e^{-2\pi i \lambda x} \left(\frac{1}{2}(-is+|b|^2) +1\right)^{-\lambda}
\end{align*}
A comparison with the corresponding calculation in Section~\ref{sec:nilpotent} shows that $D_\lambda\in L^2_{\xi_\lambda}(N(n,k)\cdot z_0)$ if $\lambda > \frac{n+1-k}{4}$.  In particular, this holds for all $\lambda > n$.

Furthermore, one sees that, if $h=(t,s,b,x), k=t',s',b',x'\in N(n)$, then
\begin{align*}
    R_\lambda(h,k)  = &\Dlambda(h) (1-\langle h\cdot z_0, k\cdot z_0\rangle )^{-\lambda} \overline{\Dlambda(k)} \\
           = &  e^{-2\pi i \lambda (x-x')}\left(\frac{1}{2}\left(-is+|b|^2\right) +1\right)^{-\lambda} \left(\frac{1}{2}\left(is'+|b'|^2\right) +1\right)^{-\lambda} \\
            & \cdot \Bigg(1-\frac{1}{\frac{1}{2}(-is + |b|^2)+1} \frac{1}{\frac{1}{2}(is' + |b'|^2)+1} \cdot \\
            & \cdot \bigg(\frac{1}{2k}\langle t,t'\rangle+\langle -ib,-ib'\rangle+\bigg( \frac{1}{2}(is - |b|^2)  \bigg) \bigg( \frac{1}{2}(-is' - |b'|^2)  \bigg)   \bigg)\Bigg)^{-\lambda} \\
            =  &  e^{-2\pi i \lambda(x-x')} \Bigg (\left(\frac{1}{2}\left(-is+|b|^2\right) +1\right) \left(\frac{1}{2}\left(is'+|b'|^2\right) +1\right) \\
           & - \langle t,t'\rangle-\langle b, b' \rangle-\left( \frac{1}{2}(is - |b|^2)  \right) \left( \frac{1}{2}(-is' - |b'|^2)  \right) \Bigg )^{-\lambda} \\
           = & e^{-2\pi i \lambda(x-x')} \left( \frac{1}{2} (-i(s-s') + |b-b'|^2)+1 -\frac{1}{2k}\sum_{i=1}^k t_i (t'_i)^{-1}\right)^{-\lambda}
  \end{align*}

As before, we note that if $f\in L^2_{\chi_\lambda}(N(k,n)\cdot z_0)$, then $\widetilde{f}\in L^2(N(k,n)\cdot z_0)$, where \[\widetilde{f}(t,s,b) = f((t,s,b,x)\cdot z_0)\chi_\lambda(x) = f((t,s,b,x)\cdot z_0) e^{2\pi i\lambda x}\]
for any $(t,s,b,x)\in \widetilde{P(n)}$.

 By Lemma~\ref{lem:kernelR}, we have that
\begin{align*}
  RR^*f(t,s,b,x)  = & \int_{\T^k\times\R\times\R^{n-1-k}} f(t',s',b') e^{-2\pi i \lambda(x-x')} \\
      & \cdot \left( \frac{1}{2} (-i(s-s') + |b-b'|^2)+1-\frac{1}{2k}\sum_{i=1}^k t_i (t'_i)^{-1} \right)^{-\lambda}dt'ds'db'  \\
     = & e^{-2\pi i x }(\widetilde{f} * \phi_{N(k,n)})(t,s,b)
\end{align*}
for all $f\in L^2(N(k,n)\cdot z_0)$, where $\phi_{N(k,n)}\in L^\infty(N(k,n))$ is defined by:
\[ \phi_{N(k,n)}(t,s,b)= 2^\lambda \left( -is + |b|^2+2 -\frac{1}{k}\sum_{i=1}^k t_i\right)^{-\lambda}.
\]
As before, one can show that $\phi_{N(k,n)}\in L^1(N(k,n)\cdot z_0)$.

One can check that, taking the Fourier transform first in the ``$s$'' variable, one obtains $\cF_s(\phi_{P(n)})(t,y,b)=0$ for $y<0$, and for $y>0$, one has:
\begin{align*}
   \cF_s (\phi_{N(k,n)}) (t,y,b) &= \frac{2^{\lambda}}{\Gamma(\lambda)} \sqrt{2\pi} y^{\lambda-1}  \exp\left(-\big(2+|b|^2- \frac{1}{k} \sum_{i=1}^{k} t_i\big)y\right) \\
   &= \frac{2^{\lambda+1/2}\sqrt{\pi}}{\Gamma(\lambda)} y^{\lambda-1} e^{-2y}e^{-|b|^2 y} \sum_{m=0}^{\infty} \frac{y^{m}}{m!}\big(\frac{1}{k} \sum_{i=1}^{k} t_i \big)^{m} \\
   &= \sum_{m=0}^{\infty}\frac{2^{\lambda+1/2}\sqrt{\pi}}{m!\Gamma(\lambda)} \cdot \frac{y^{\lambda+m-1}}{k^m} e^{-2y}e^{-|b|^2y}  \\
   &\quad\times \sum_{\substack{m_1,\ldots, m_{k} \in\N_0 \\ m_1 +\cdots + m_{k} =  m}} \frac{m!}{m_1!\cdots m_{k}!} t_1^{m_1}\cdots t_n^{m_k}.
\end{align*}

Taking the Fourier transform in the ``$b$'' and ``$t$'' variables, we obtain:
\[
    \widehat{\phi_{N(k,n)}}(\alpha,y, \xi) = \frac{2^{\lambda-(n-k)/2}\sqrt{\pi}}{\Gamma(\lambda)
     k^{|\alpha|}}    y^{\lambda+|\alpha|-1-(n-k)/2}e^{-2y}e^{-|\xi|^2/8y} \frac{1}{\alpha_1!\cdots \alpha_k!},
\]
where $\alpha\in\N^k$, $y>0$, and $\xi\in\R^{n-1-k}$.  If  $y\leq 0$ or else $\alpha\in\Z^k$ but $\alpha_i<0$  for some $1\leq i\leq k$, then $\widehat{\phi_{N(k,n)}}(\alpha,y,\xi)=0$.

\section{Spectrum of Toeplitz operators with symbols invariant under restriction-injective subgroups}
\label{sec:spectrum}
In this section, we show in Theorem~\ref{thm:mainspectraltheorem} that if a Toeplitz operator for a complex bounded symmetric domain $G/K$ has a symbol that is invariant under a subgroup $H\le G$ with restriction-injective orbit $H\cdot z_0$ such that $h \mapsto j_\lambda(h,z_0)$ lies in $L^2_{\chi_\lambda}(H\cdot z_0)$, then it can be written as a convolution operator using the Segal-Bargman transform that was defined in Section~\ref{subsec:restrictionprinciple}. These results are stated for the case of $G/K= \B^n$, but hold for any complex bounded symmetric domain $G/K$ and any $\lambda\in \R$ such that the holomorphic discrete series representation $\pi_\lambda$ can be defined, since only the basic properties of the reproducing kernel and the cocycle condition for $j_\lambda$ are used. Finally, we apply the results of the previous section to calculate the spectrum of a Toeplitz operator $T^{(\lambda)}_\varphi:\HBlam\rightarrow\HBlam$ with $H$-invariant symbol $\varphi \in L^{\infty}(\B^n)^H$, where $H$ is a maximal abelian subgroup of $\SU(n,1)$.

We begin by supposing that $H$ is any subgroup of $G$ for which the results of Section~\ref{subsec:restrictionprinciple} hold:  that is, we suppose that at least one orbit $H\cdot z_0$ is restriction injective and the map $D_\lambda: h\mapsto j_\lambda(h,z_0)$ lies in $L^2_{\chi_\lambda}(H\cdot z_0)$.  For all  $f\in \Blam$ that lie in the domain of $(\sqrt{RR^*})^{-1}$, we can write
\begin{align*}
 U_\lambda^* T^{(\lambda)}_\varphi U_\lambda f & = (\sqrt{RR^*})^{-1} R T_\varphi^{(\lambda)}R^* (\sqrt{RR^*})^{-1} f\\
    & = (RR^*)^{-1}  RT_\varphi^{(\lambda)} R^* f,
     \end{align*}
since all $\widetilde{H}$-intertwining operators on $\Blam$ commute because the representation is multiplicity free.  One can explicitly write $RT_\varphi^{(\lambda)} R^*$ as a convolution operator using the following theorem:

\begin{theorem}
\label{thm:mainspectraltheorem}
Let $H\subseteq G$ such that for some $z_0\in \B^n$, the orbit $H\cdot z_0$ is restriction injective and the function $h\mapsto j_\lambda(h,z_0)$ lies in $L^2_{\chi_\lambda}(H\cdot z_0)$, as in Section~\ref{sec:toeplitzintro}.  If $\varphi \in L^\infty(\B^n)^H$ is an $H$-invariant symbol on $\B^n$, then the operator
\[
     R  T^{(\lambda)}_\varphi R^*: \Blam\rightarrow \Blam
\]
is given by
\[
    R  T^{(\lambda)}_\varphi R^* f = f * \nu_\varphi
\]
for all $f\in L^2_{\chi_\lambda}(H\cdot z_0)$, where $\nu_\varphi: H/H_{z_0} \rightarrow \C$ is defined by:
\[
   \nu_\varphi (h\cdot z_0) = j_\lambda(h,z_0)\langle \varphi\, K_{z_0} , K_{h\cdot z_0} \rangle _{L^2(\B^n,\mu_\lambda)} = \int_{\B^n} \varphi(z) K_{z_0}(z) \overline{K_{h\cdot z_0}(z)} dz
\]
\end{theorem}

To prove this result we will need the following lemma:
\begin{lemma}
\label{lemma:actionOnReprodKern}
For all $g\in G$ and $z\in\B^n$,
\[
     \pi_\lambda(g)K_z = \overline{j_\lambda(g^{-1},z)} K_{g^{-1}\cdot z}.
\]
\end{lemma}
\begin{proof}[Proof (of lemma)]
Note that, for all $f\in \HBlam$, one has that
\begin{align*}
   \langle f, \pi(g) K_z\rangle  & = \langle \pi(g^{-1})f, K_z \rangle \\
      & = (\pi(g^{-1})f)(z) \\
      & = j_\lambda(g^{-1},z)f(g^{-1}\cdot z) \\
      & = \langle f, \overline{j_\lambda(g^{-1},z)} K_{g^{-1}\cdot z}\rangle.
      \end{align*}
The result then follows since this equality holds for all $f\in \HBlam$.
\end{proof}

\begin{proof}[Proof (of theorem)]
We see that $RT_\varphi^{(\lambda)}R^* f(h) = D_\lambda(h) T_\varphi^{(\lambda)}R^*f(h\cdot z_0)$ for
all $h\in \widetilde{H}$.  Next, we note that \begin{align*}T_\varphi^{(\lambda)} R^*f (h\cdot z_0) &  = 
\langle \varphi \, (R^*f), K_{h\cdot z_0} \rangle_{L^2(\B^n,\mu_\lambda)} \\
  & = \int_{\B^n} \varphi(z) R^*f(z) \overline{K_{h\cdot z_0}(z)} d\mu_\lambda(z)
\end{align*}
Meanwhile, we recall that
\[
R^*f(z) = \langle  f,RK_z\rangle_{\Blam}
    = \int_{H/H_{z_0}} f(k) \overline{D_\lambda(k) K_z(k\cdot z_0)} dk
\]
for all $z\in\B^n$, where we are implicitly using that $k\mapsto f(k)\overline{D_\lambda(k)}$ factors to a well-defined function on $H/H_{z_0}$.  Combining these two identities, using that $K_z(w) = \overline{K_w(z)}$, and applying Fubini's theorem yields:
\begin{align*}
RT_\varphi^{(\lambda)}R^* f(h) & = D_\lambda(h) T_\varphi^{(\lambda)}R^*f(h\cdot z_0)
    \\& = D_\lambda(h)\int_{\B^n}  \varphi(z)\int_{H/H_{z_0}}  f(k)  \overline{D_\lambda(k)} K_{k\cdot z_0}(z) \, dk\, \overline{K_{h\cdot z_0}(z)}\, d\mu_\lambda(z)
    \\& = \int_{H/H_{z_0}} f(k) \int_{\B^n} \overline{j_\lambda(k,z_0)}\varphi(z) K_{k\cdot z_0} j_\lambda(h,z_0)\overline{K_{h\cdot z_0}}dk d\mu_\lambda(z)
    \\& = \int_{H/H_{z_0}} f(k)  \left\langle \overline{j_\lambda(k,z_0)}\varphi\, K_{k\cdot {z_0}}, \overline{j_\lambda(h,z_0)}K_{h\cdot z_0} \right\rangle_{L^2(\B^n,\mu_\lambda)} dk.
    \end{align*}
Next, we note that, for all $k\in\widetilde{H}$, one has that:
 \begin{align*}\pi_\lambda(k)(\varphi\, K_{k\cdot z_0})(z) & = j_\lambda(k^{-1},z)\varphi(k^{-1}\cdot z) 
 K_{k\cdot z_0}(k^{-1}\cdot z) \\
  &  = \varphi(k^{-1}\cdot z) \pi_\lambda(k) (K_{k\cdot z_0})(z) \\ & =  \overline{j_\lambda(k^{-1},k\cdot z_0)}\varphi(z)K_{z_0}(z),
 \end{align*}
where we use the $\widetilde{H}$-invariance of $\varphi$ in the last line.  Thus, by applying the
unitary  operator $\pi_\lambda(k)$ to both sides of the inner product and using
Lemma~\ref{lemma:actionOnReprodKern} and the cocycle relations for $j_\lambda$, we see that
\begin{align*}
&  \left\langle \overline{j_\lambda(k,z_0)}\varphi\, K_{k\cdot {z_0}}, \overline{j_\lambda(h,z_0)}K_{h\cdot z_0} \right\rangle_{L^2(\B^n,\mu_\lambda)}
 \\& = \left\langle \overline{j_\lambda(k,z_0)}\pi_\lambda(k)(\varphi\, K_{k\cdot {z_0}}), \overline{j_\lambda(h,z_0)}\pi_\lambda(k)K_{h\cdot z_0} \right\rangle_{L^2(\B^n,\mu_\lambda)}
 \\ & = \langle \overline{j_\lambda(k,z_0)j_\lambda(k^{-1},k\cdot z_0)}\varphi\, K_{z_0}, \overline{j_\lambda(h,z_0)j_\lambda(k^{-1},h\cdot z_0)}K_{k^{-1}h\cdot z_0} \rangle
 \\& = j_\lambda(k^{-1}h,z_0)\langle \varphi\, K_{z_0}, K_{k^{-1}h\cdot z_0}\rangle
\end{align*}
so that
\[RT_a^{(\lambda)}R^* f(h) = \int_{H/H_{z_0}} f(k) j_\lambda(k^{-1}h,z_0)\langle \varphi\,
K_{z_0}, K_{k^{-1}h\cdot z_0} \rangle dk. \]
\end{proof}

Combining this result with those of Section~\ref{sec:maxabelian}, we can now diagonalize the Toeplitz operators with $H$-invariant symbols as follows when $H\subseteq \SU(n,1)$ is a maximal abelian subgroup:
\begin{corollary}
Let $H$ be a maximal abelian subgroup of $\SU(n,1)$, and fix $z_0$ as in Section~\ref{sec:maxabelian}.
Let $A=\operatorname{supp} \widehat{\phi_H}$, where the Fourier transform is taken over $H/H_{z_0}$.
Let $A = \{\alpha\in \widehat{H/H_{z_0}} \mid \widehat{\phi_H}(\alpha) \neq 0\}$. We define the modified Fourier
 transform $\cF:\Blam\rightarrow L^2(A)\subseteq L^2(\widehat{H/H_{z_0}})$ by setting $\cF f(\alpha) = \cF_{H/H_{z_0}} \widetilde{f}(\alpha)$, where $\widetilde{f} = \chi_{-\lambda} f$ as before.

Let $\varphi\in L^\infty(\B^n)^H$ be an $H$-invariant symbol.  Then
\[
     \cF U_\lambda^{-1} T_\varphi U_\lambda \cF^{-1} \omega(\alpha) = 
     \frac{\widehat{\nu_\varphi}(\alpha)}{\widehat{\phi_H}(\alpha)} \omega(\alpha)
\]
for all $\alpha\in A$ and $\omega\in L^2(\widehat{H/H_{z_0}})$ such that $\operatorname{supp} \omega\subseteq \operatorname{supp} \widehat{\phi_H}$
\end{corollary}

 \end{document}